\documentclass[12pt]{article}
\usepackage{amsmath,amsthm}
\usepackage{amsfonts}
\usepackage{amssymb}
\usepackage{latexsym}   
\usepackage{graphicx,enumerate}
\usepackage{color}
\usepackage{hyperref,subfig}
\usepackage{algorithm,algorithmic}

\usepackage{vmargin,fancyhdr}   % margins
\usepackage{srcltx}

\newcommand{\card}[1]{\left|#1\right|}

\newcommand{\beql}[1]{\begin{equation}\label{#1}}
\newcommand{\eeq}{\end{equation}}

\newcommand{\Floor}[1]{{\left\lfloor{#1}\right\rfloor}}

\newcommand{\N}{{\mathbb N}}

\newtheorem{lemma}{Lemma}
\newtheorem{theorem}{Theorem}

\newcommand{\abs}[1]{\left\lvert #1 \right\rvert}

\newcommand{\old}[1]{}

\newcommand{\tref}[1]{Theorem~\ref{t.#1}}

\newcommand{\lref}[1]{Lemma~\ref{l.#1}}

\setpapersize{USletter}
\setmarginsrb{.75in}{.5in}        % left, top
             {.75in}{.5in}        % right, bottom
             {.25in}{.25in}     % headheight, headsep
             {.25in}{.5in}      % footheight, footskip
\newcommand{\hide}[1]{}

\newcommand{\beq}[1]{\begin{equation}\label{#1}}

\newcounter{rot}%\addtocounter{rot}{1}, \therot

\def\a{\alpha} \def\b{\beta}  
\def\e{\epsilon}    
\def\G{\Gamma}  
\def\z{\zeta}     
   \def\p{\pi}

\newtheorem{claim}{Claim}

\newcommand{\rdup}[1]{{\left\lceil #1 \right\rceil }}

\newcommand{\brac}[1]{\left(#1\right)}

\newcommand{\set}[1]{\left\{#1\right\}}

\def\E{\mbox{{\bf E}}}
\def\Pr{\mbox{{\bf Pr}}}

\newcommand{\ignore}[1]{}
\begin{document}
\title{Walker-Breaker games}
\author{Lisa Espig 
\thanks{Department of Mathematical Sciences,
Carnegie Mellon University, Pittsburgh PA15213, USA.
Research supported in part by NSF grant ccf1013110,
\hbox{Email}~{\small\texttt{lespig@andrew.cmu.edu.}}}
\and
Alan Frieze
\thanks{Department of Mathematical Sciences, 
Carnegie Mellon University, Pittsburgh PA15213, USA.
Research supported in part by NSF grant ccf1013110,
Department of Mathematical Sciences,
\hbox{Email}~{\small\texttt{alan@random.math.cmu.edu.}}}
\and
Michael Krivelevich\thanks{School of Mathematical Sciences,
Raymond and Beverly Sackler Faculty of Exact Sciences, Tel Aviv
University, Tel Aviv, 69978, Israel.
Research supported in part by a USA-Israel BSF grant and by a grant 
from the Israel Science Foundation.
\hbox{Email}~{\small\texttt{krivelev@post.tau.ac.il.}}}
\and Wesley Pegden\thanks{Department of Mathematical Sciences, 
Carnegie Mellon University, Pittsburgh PA15213, USA.
Research supported in part by NSF grant,
\hbox{Email}~{\small\texttt{wes@math.cmu.edu.}}}}

\maketitle
\begin{abstract}
We introduce and analyze the \emph{Walker-Breaker} game, a variant of Maker-Breaker games where Maker is constrained to choose edges of a walk or path in a given graph $G$, with the goal of visiting as many vertices of the underlying graph as possible.
\end{abstract}

\section{Introduction} 
Maker-Breaker games were intoduced by Erd\H{o}s and Selfridge
\cite{ES73} as a generalisation of Tic-Tac-Toe. Since then there have
been many results on variations on this theme. In a standard version, played on the
complete graph $K_n$, Maker and Breaker take turns acquiring edges, with 
Maker trying to build a particular
structure (e.g., a clique) in his own edges, and with Breaker trying to prevent this. See the recent
book by Beck \cite{Beck} for a comprehensive analysis of Maker-Breaker games.

We consider the following variant on the standard Maker-Breaker
game. In this variant, the \emph{Walker-Breaker game}, the ``Walker'' acquires the edges of a walk consecutively; i.e., at any given moment of the game we have her positioned at some vertex $v$ of a graph $G$ and on her
turn, she moves along an edge $e$ of $G$ that is (i) incident with $v$ and
(ii) has not been acquired by Breaker. If she has not already acquired
$e$, then she is now considered to have acquired it. On Breaker's
move, he can acquire any edge not already owned by Walker. In some
cases we will allow him to acquire $\b$ edges in one move; in this case the \emph{bias} of the game is $1:\b$.

In this paper, we consider Walker-Breaker games where Walker's goal is to visit as many vertices as she can. Breaker's goal is
to reduce the number of vertices that she visits.  The game ends when there is no path from Walker's current position to an unvisited vertex along edges not acquired by Breaker.

We also consider a variant of this game, the \emph{PathWalker-Breaker
  game}, in which Walker cannot revisit any previously visited
vertices; this game ends when there is no path from Walker's current
position to an unvisited vertex along edges not acquired by Breaker,
and vertices not previously visited by Walker.  (Obviously, Walker can
visit at least as many vertices in the Walker-Breaker game on a graph
as in the PathWalker-Breaker game on the same graph).  In this
situation, we sometimes refer to Walker as PathWalker to avoid ambiguity.

In a fictional scenario, Walker represents a missionary who is
traversing a network, trying to convert as many people ($\equiv$
vertices) to his beliefs. Breaker represents the devil, whose only way
to block Walker is to burn untraversed edges of the network.

%We will examine some scenarios and estimate how well Walker can do in each of them. Our results are independent of who goes first in this game.

\smallskip 

Our first Theorem can be seen as a strengthening of the result of Hefetz, Krivelevich,  Stojakovi\'c and Szab\'o \cite{HKSS}, that in a Maker-Breaker game on $K_n$, Maker can construct a Hamilton path in $n-1$ moves. 

\begin{theorem}
\label{t.rw11}
Under optimum play in the $1:1$ PathWalker-Breaker game on $K_n$ $(n>5)$, PathWalker visits all but two vertices.
\end{theorem}

\noindent The bias has a substantial effect on the PathWalker-Breaker game:
\begin{theorem}\label{t.rw1b}
Under optimum play in the $1:\beta$ PathWalker-Breaker game on $K_n$ for $1<\beta=O(1)$, PathWalker visits all but $s$ vertices for $c_1\log n\leq s \leq c_2\log n$, for constants $c_1,c_2$ depending on $\beta$.
\end{theorem}

\noindent In the Walker-Breaker game, the effect of bias is not so drastic:
\begin{theorem}\label{t.prw1b}
Under optimum play in the $1:\beta$ Walker-Breaker game on $K_n$ $(n>\b^2)$, Walker visits $n-2\beta+1$ vertices.  Here $1\leq \beta=O(1)$.
\end{theorem}

\noindent For the sake of a graph which is not complete, consider the cube $Q_n$, which is the graph on the vertex set $\{0,1\}^n$, where two strings are adjacent iff they have Hamming distance 1.  
\begin{theorem}\label{t.prwQ}
In optimum play in the Walker-Breaker game on $Q_n$, Walker visits at least $2^{n-2}$
vertices, and at most $2^{n-1}$ vertices. 
\end{theorem}

\noindent Finally, we consider a one-player game the
\emph{Random-Walker game}, in which the moves of Walker are made
according to a random walk on the edges not acquired by Breaker.
Breaker acquires one edge per move and he has the goal of minimizing the typical number of vertices visited by RandomWalker.  The game ends when there is no path between the position of RandomWalker and an unvisited vertex along edges not acquired by Breaker.
\begin{theorem}\label{t.rwL}
If $G$ has minimum co-degree at least $\a n$ for some absolute constant $\a>0$
then under optimum play (by Breaker), RandomWalker visits all but
at most $c\log n$ vertices of $G$ w.h.p., for a constant $c$ depending
on $\alpha$.
\end{theorem}

\begin{theorem}\label{t.rwU}
If $G$ has minimum degree at least $\a n$ for some absolute constant $\a>0$
then under optimum play (by Breaker), RandomWalker visits at most $n-c\log n$ vertices w.h.p., for any constant $c<\alpha$.
\end{theorem}

\subsection{Some notation}
We let $V_t$ (resp. $U=U_t$) denote the set of vertices that have been
visited (resp. not visited) by Walker after Walker has made $t$ moves.
Walker is at vertex $v_t$ after $t$ moves.
The graph induced by Breaker's edges is denoted
by $\G_B$ and the graph induced by Walker's edges is denoted by
$\G_W$.

%\medskip
%$\n_W$ denotes the number of vertices visited by Walker.
\section{Proof of \tref{rw11}}\label{anb}
Here at time $t$ the graph $\G_W$ is a path $P_t$. To show that who
goes first does not matter, we assume that Breaker goes first for any
lower bound on the number of visited vertices, and that Walker goes first for any upper bound on the number of visited vertices.
\subsection{Lower bound}

Walker's strategy is
as follows: 
If $\abs{U_t}>2$ and Breaker chooses $f_t$ and $f_t\cap V_{t-1}=\emptyset$
then Walker moves to $v_t\in f_t$. Otherwise, Walker moves 
to an arbitrary vertex. So long as Walker is able to follow this strategy, we will have after each Walker move that:
\beq{1}
\text{Every Breaker edge contains a member of $V_t$.}
\eeq 
We now check that this strategy is feasible for $\abs{U_t}>2$.  We begin with the first type of Breaker's edge, which is disjoint from $V_t$.  Fix $t$ and let $v_{t-1}=x$ and $v_t=y$.
Suppose that Breaker chooses an edge $(b_1,b_2)$ where $b_1,b_2\notin V_t$ and such that
for $i=1,2$, $(y,b_i)$ is a Breaker edge. Assume that this
is the first time this situation happens. Suppose next that $(y,b_i)$
is the $s_i$-th edge
chosen by Breaker. Assume that $s_1<s_2$.  We now have a contradiction 
to \eqref{1} after the choice of $x$. For after $x$ is
chosen, $(y,b_1)$
is a Breaker edge that does not contain a member of $V_{t-1}$.

We now consider the case where Breaker's edge is incident with $V_t$.  \eqref{1} implies that Breaker's choice is at most the second edge between $v_t$ and $U_t$.  In particular, $\abs{U_t}>2$ implies that Walker can move to an unvisited vertex, and Walker will succeed at visiting all but 2 vertices of the graph.

\subsubsection{Upper bound}
Breaker plays arbitrarily until his move at time $n-4,$ when $\abs{U_{n-4}}=4$.  In his next two moves he chooses the two edges of a matching in $U_{n-4}$.  After these two moves (with one Walker move in between), it is again Walker's turn, and 3 unvisited vertices remain.  Regardless of which vertex of $U_{n-3}$ Walker might move to next, that vertex will already be adjacent along one of Breaker's 2 matching edges to a vertex in $U_{n-2}$; thus, with one additional move, Breaker will ensure that both edges from $v_{n-2}$ to $U_{n-2}$ are occupied by Breaker.
\section{Proof of \tref{rw1b}}
\subsection{Lower bound}\label{AS}
We assume that Breaker goes first and describe Walker's strategy.
We suppose that Walker is at some vertex $x$ and describe the next
sequence of moves B,W,B,W (Breaker,Walker,Breaker,Walker.) We call such a sequence a {\em round}. 
We keep track of two sets $L,R$ that partition the set of unvisited vertices $U$. 
Let $\b_R(v)$ be the number of Breaker edges $(v,z),z\in R$. 
At the outset of the game, $L=\varnothing$.  Before Breaker's first
move in each round, we move vertices satisfying $\b_R(v)\geq \alpha R$
from from $R$ to $L$ one by one, (updating $R$ each time), until no such vertices remain, for $\alpha=\frac{1}{3(\beta+1)}$.

\medskip

We describe Walker's strategy for a round as follows.  
Suppose that Breaker has made his first move of the round and 
let $R=\set{w_1,w_2,\ldots,w_r}$ at the end of this move.
We assume that $\b_R(w_i)\geq \b_R(w_{i+1}),1\leq i<r$. 

For his first move of the round, Walker moves to a vertex $z\in R$ such that none
of $(z,w_j),1\leq j\leq \b+1$, is a Breaker edge.   Breaker's response consists of just $\beta$ edges; Walker's final move is to move from $z$ to one of the vertices $w_j$ $1\leq j\leq \b+1$.

We will prove that it is possible to follow this strategy until most vertices have been visited.  This proof is based on two ingredients:
\begin{claim}
\label{c.Lmoves}
There is a constant $c_\beta$ such that so long as Walker follows this strategy and so long as $\abs R>c_\beta\log n$, at most $\beta$ vertices are moved from $R$ to $L$ in any given round.
\end{claim}
\begin{claim}
\label{c.smallL}
There is a constant $C_\beta$ such that so long as  Walker follows this strategy, we will have $\abs{L}\leq C_\beta\log n$.
\end{claim}
Let us first see how these two claims imply that Walker can follow this strategy until she has visited all but $c_2\log n$ vertices for some $c_2$ depending on $\beta$.

We first check that Walker can always move to a suitable intermediate
vertex $z$. 
After Breaker's opening move of the round, at most $(\b+1)\a|R|+\beta=|R|/3+\beta$ vertices of $R$ can be Breaker neighbors of $w_1,w_2,\ldots,w_{\b+1}\in R$. 
Moreover, the fact that $x$ was in $R$ at the beginning of the previous round, together with Claim \ref{c.Lmoves}, means that if $\abs R>c_\beta \log n$, then $x$ 
has at most $\beta+\alpha(\abs{R}+\beta)\leq \beta+\abs R / 3$ neighbors in $R$, leaving at least $ \frac 1 {3} \abs{R}-2\beta$
choices
for $z$. So, Walker will be able to move to such a $z$ as long as $\abs R>\max(6\b,c_\beta \log n)$.  For $\beta=O(1)$, Claim \ref{c.smallL} now implies that Walker can follow this strategy until all but $(C_\beta+c_\beta)\log n$ vertices have been visited.

It remains to prove Claims \ref{c.Lmoves} and \ref{c.smallL}.  We do this via a simpler \emph{box game}.
\subsubsection{Box Game}
We analyze Walker's strategy via a {\em Box Game}, similar to the Box Game of Chv\'atal and
Erd\H{o}s \cite{CE}. This is not really a game, as there is only one
player whom we call BREAKER. Any move by Breaker in the Walker-Breaker
game willbe modelled by a BREAKER move in this Box Game.

Consider a sequence $b_1\geq b_2\geq \cdots \geq b_n$ of non-negative integers.  The Box game is played as follows:
At the beginning of each turn of the game, BREAKER has a \emph{loss} phase, in which he may at his option, delete terms with value at least 
$\a=\frac{1}{3(\b+1)}$ times the remaining number of terms in the
sequence.   (At
any point, the sets of remaining, lost, and deleted terms of the
original sequence form a partition of the terms of the original sequence.)

Following the loss phase,  BREAKER increases each of the first $\b$ terms of the
sequence by an amount up to $4\b$. In addition he also increases terms
$b_i$ for $i>\b$ by a total amount up to $4\b$. After this he deletes one of
the currently largest $\b+1$ terms $b_1,b_2,\dots, b_{\b+1}$ of the
sequence, and up to one other term from anywhere in the sequence.

The relevance of this game stems from the following:
\begin{claim}
\label{c.boxgraph}
If Breaker can play the PathWalker-Breaker game on a graph with $n$ vertices against
Walker which is following the strategy described earlier such that
$\abs{L}$ increases by $\ell_t$ in each round $t$, then in this Box game, beginning with the all
zero's sequence of length $n$, BREAKER can force that $\ell_t$ terms become
lost in turn $t$.
\end{claim}
\begin{proof}
At any point, the sequence $(b_i)$ represents Breaker
degrees of vertices not yet visited by Walker and not in $L$, in
decreasing order.  In any round, Breaker places $2\b$ edges in the
graph, which increases the total degrees by $4\b$; thus, BREAKER can
produce the exact same resulting (degree) sequence with an allowed alteration of the terms of the sequence.  In any round, Walker will visit two vertices, the
second of which is a vertex of among the highest possible $\b+1$
degrees; this corresponds to the deletion of the two terms on each
turn of the box game.  If in each loss phase of the box game, BREAKER loses as many terms as possible, then BREAKER will lose exactly as many terms as vertices enter $L$ in the the PathWalker-Breaker game.
\end{proof}
Note that our definition of the Box Game allows much
more freedom to BREAKER than is necessary for Claim \ref{c.boxgraph}. This extra freedom
does however simplify the analysis, by enabling us to
decouple consideration of the first $\b$ boxes from the rest.  In particular, call a term in the box game sequence a \emph{tail term} if it is not among the $\b$ largest. We will prove Claims \ref{c.Lmoves} and \ref{c.smallL} by proving the following lemma regarding the box game:
\begin{lemma}
\label{l.boxtail}
There is a constant $A_\beta$ such that after any number of steps $t$ in the box game, the tail terms are all at most $A_\beta \log t$.
\end{lemma}
First let us observe that \lref{boxtail} implies both Claims \ref{c.Lmoves} and \ref{c.smallL}, via Claim \ref{c.boxgraph}.

\begin{proof}[Proof of Claim \ref{c.Lmoves}]  Let
  $c_\beta=A_\beta/\alpha$, and suppose Breaker can play the
  PathWalker-game on a graph with $n$ vertices against Walker which
  is following the strategy described earlier, and achieve that on
  some turn, more than  $\beta$ vertices become lost i.e. move from
  $R$ to $L$.  Claim
  \ref{c.boxgraph} implies that he can play the box game and achieve
  that on some turn, more than $\beta$ terms $b_j$ become lost, which
  would be a contradiction since $\abs R>\frac{A_\beta}{\alpha} \log
  n$ at the beginnning of the turn and $b_j\leq A_\beta \log t\leq A_\beta \log n$ for $j>\beta$ implies that $b_j<\alpha \abs R$ for $j>\beta$; in particular, only the $\beta$ largest terms can become lost on any given turn.
\end{proof}
\begin{proof}[Proof of Claim \ref{c.smallL}]
It suffices to prove that the claim holds so long as $\abs R >
2c_\beta \log n$ (with $c_\beta=A_\beta/\alpha$, as before), since
Claim \ref{c.smallL} will then remain true even if all remaining
vertices in $R$ become lost.  In particular, we may assume that in any
given round, only vertices from among the $\beta$ maximum Breaker-degree vertices of $R$ become lost.

Using Claim \ref{c.boxgraph}, we carry out our analysis in the simplified box game.  At the beginning of a turn, some terms $b_i$ for $i\leq \b$ may become lost.  Suppose this ``box'' became one of the $\b$ largest (for the last time) at turn $t_0$ when $r_0$ terms remained, and becomes lost at turn $t_1=t_0+k$, when a total of $r_1$ terms remain.  \lref{boxtail} implies that the box had at most $A_\b\log t_0$ balls at turn $t_0$.  To become lost at time $t_1$ requires the box to have at least $\alpha r_1$ balls; thus, we have that
\[
\alpha r_1-A_\b \log t_0\leq 4\b k.
\]
Since $r_1\geq r_0-k\b$, we have
\[
k\geq \frac{\alpha r_1-A_\b \log t_0}{4\b}\geq \frac{\alpha r_1}{8\beta}\geq \frac{\alpha(r_0-\b k)}{8\b},
\]
since $r_1\geq 2c_\b \log n\geq 2c_\b \log t_0$.  In particular,
\begin{equation}
\label{e.mink}
k\left(1+\frac{\alpha}{8}\right)\geq \frac{\alpha r_0}{8\beta}\implies k\geq \frac{\alpha r_0}{10\b}.
\end{equation}
In particular, when a term enters the largest $\b$ for the last time, it takes at least a number of steps $k$ which is a constant fraction of the number $r_0$ of remaining terms when it entered, before it can become lost.   
Once it becomes lost, say, when there are $r_1$ remaining terms, some
other term enters the largest $\b$ terms, and to become lost this term requires at
least a number of terms to become lost which is a constant fraction of $r_1$.
Continuing in this manner, we see the terms $r_1,r_2,\dots,$ of this
sequence must satisfy (with $k$ as in \eqref{e.mink})
$r_{i+1}\leq r_i-k\leq r_{i}(1-\frac{\alpha}{10\beta})$ by \eqref{e.mink}, since
at least one term is deleted on each turn of the box game.  In particular, with $r_0=n$, we have that this sequence can have at most $\frac{\log n}{\log\left(\frac{10\b}{10\b-\alpha}\right)}$ terms.  We can have $\beta$ different such sequences producing lost terms (one for each of the $\beta$ initially largest terms), giving a max of 
\[
\frac{\beta}{\log\left(\frac{10\beta}{10\beta-\alpha}\right)}\log n
\]
lost terms produced.
\end{proof}

We will use the following Lemma to prove \lref{boxtail}:
\begin{lemma}
\label{l.dominating}
Suppose that $e_1\geq e_2\geq \cdots \geq e_s$ and $b_1\geq b_2\geq \cdots \geq b_r$ 
and  are two states of the box game
where $s\leq r$ and $e_i\leq b_i$ for all  $1\leq i\leq s$, and $f:\N\to \N$ is an arbitrary function. %,b_2=e_2,\dots,b_\beta=e_\beta$, and $b_j\geq e_j$ for $j>\beta$.  
If BREAKER has a strategy in the first to ensure that for some $t$, some tail term has value $\geq f(t)$,  then he has a strategy in the second  to ensure that for some $t$, some tail term is $\geq \ell$ by the turn $t$.
\end{lemma}
\begin{proof}  BREAKER simply mimics his strategy for the sequence $\{e_i\}$ with the sequence $\{b_i\}$.  Terms deleted or lost for the game on the first sequence are deleted or lost, respectively, for the game on the second sequence.  (Note that the freedom BREAKER has to choose not to lose terms is important here.)
\end{proof}

\noindent We are now ready to prove \lref{boxtail}.
\begin{proof}[Proof of \lref{boxtail}]
\lref{dominating} implies that it suffices to prove \lref{boxtail} for
the case where, on each turn, after the loss phase, only one term is deleted by BREAKER, and this term is the $(\b+1)$'st largest term.  Under this assumption, the terms $b_{\b+1},b_{\b+2},\dots$ are reproducing the classical box game of Chv\'atal and Erd\H{os} \cite{CE}, see 
also Hamidoune and Las Vergnas \cite{HL}, since the largest of these terms is deleted on each turn.  In particular, a simple potential function argument shows that $b_{\b+1}\leq A_\b\log t$ throughout, for $A_\b=\frac 1 {\log\left(\frac {4\b} {4\b-1}\right)}$.
\end{proof}

\subsection{Upper bound}\label{ub1}
Breaker's strategy is as follows: Breaker
chooses a vertex $w_1\notin \set{v_1,v_2}$. He will spend the next $(n-1)/\b$
moves making sure that Walker cannot visit $w_1$. In a move,
Breaker claims the edge from $w_1$ to $v_t$, if necessary, plus $\b-1$ other edges incident with
$w_1$. This takes approximately $n_1=(n-1)/\b$ moves. Breaker then chooses
another unvisited vertex $w_2$ and spends approximately $n_2=(n-1-n_1)/\b$
moves protecting $w_2$. It takes only $n_2$ rather than $n_1$ moves 
because Walker cannot use $n_1$ of the edges to $w_2$, because she has
visited the other endpoint. Continuing in
this manner Breaker protects $w_k$ in $n_k$ moves where
$$n_k=\frac{n-1-(n_1+n_2+\cdots+n_{k-1})}{\b}
=\frac{n-1}{\b}\brac{1-\frac{1}{\b}}^{k-1}.$$
It follows from this that 
$$n_1+n_2+\cdots n_k=(n-1)\brac{1-\brac{1-\frac{1}{\b}}^k}.$$
Thus we can take $k=c_1\log n$ where $c_1=1/\log(\b/(\b-1))$.
This will be our value of $c_1$ in Theorem \ref{t.rw1b}.

This completes the proof of Theorem \ref{t.rw1b}.
\section{Proof of Theorem \ref{t.prw1b}}
In this section, Walker is not constrained to a path (her walk may use an edge more than once). %In the light of Theorem \ref{t.rw11}(a), we now assume that $\b>1$. 
\subsection{Lower bound}\label{lowerc}
Walker builds a tree $T$ in a depth first manner. She starts at the root $v_1$ 
at depth 0. All depth/parent/child statements are with respect to this root. A vertex $v\in T$ will have a parent $w=\p(v)$ where 
the depth of $v$ is one more than the depth of $w$. If Walker is at
vertex $x$ and there is a vertex $y\in U_t$ such that Breaker
has not claimed the edge $(x,y)$ then Walker moves to $y$. We let
$x=\p(y)$. 
Otherwise, if no such move is possible, Walker moves to $\p(x)$
and repeats the search for $y\in U$ on its next move.
The game is over when Walker finds herself at $v_1$ and all edges
$v_1$ to $U$ have been taken by Breaker.

Suppose that the game ends with $|U|=k$. Then Walker has made $2(n-k-1)$
moves. Each edge of $T$ has been traversed twice, once in a forward direction 
and once in a backwards direction. Breaker has captured at least $k(n-k)$ edges
between $T$ and $U$. We therefore have
$$k(n-k)\leq 2\b(n-k-1).$$
It follows from this that $k<2\b$. This shows that Walker visits at least
$n-2\b+1$ 
vertices.
\subsection{Upper bound}
The argument here has some similarities to that in Section \ref{ub1}.
Breaker's strategy is as follows: Assuming Walker goes first and claims an edge $\set{v_1,v_2}$, Breaker
chooses a vertex $w_1\notin \set{v_1,v_2}$. He will spend the next $(n-1)/\b$
moves making sure that Walker cannot visit $w_1$. In a move,
he claims the
edge from $w_1$ to $v_t$, if necessary, plus $\b-1$ other edges incident with
$w_1$. This takes approximately $(n-1)/\b$ moves. Then he chooses $w_2$, not
visited and protects it from being visited in the same way. He does this
for $w_1,w_2,\ldots,w_{\b-1}$. Altogether, this takes up at most $\rdup{(n-1)(\b-1)/\b}$
moves, leaving at least $\Floor{(n-1)/\b}+1$ vertices unvisited. Breaker then chooses $\b$ unvisited
vertices $y_1,y_2,\ldots,y_\b$ (possible since $n-1\geq \b^2$) and a move consists of capturing the edges 
$(v_t,y_i),\,i=1,2,\ldots,\b$. This protects $y_1,y_2,\ldots,y_\b$ and so
Walker visits at most 
$n-2\b+1$ vertices.
This completes the proof of Theorem \ref{t.prw1b}.
\section{Proof of Theorem \ref{t.prwQ}}
\subsection{Lower bound}
We use a similar argument to that in Section \ref{lowerc}. Walker builds a Depth First Search tree $T$. Again, the edges between $T$ and $U$ will all be Breaker's edges. Suppose now that $T$ has $k$ vertices. Then 
$$2(k-1)\geq e(T,U)\geq k(n-\log_2k).$$
The lower bound follows from Harper's theorem \cite{H66}. It follows that 
$$\log_2k\geq n-2+\frac{2}{k}$$
and so at least
$2^{n-2}$ vertices are visited by Walker.
\subsection{Upper bound}
Suppose that Walker goes first and assume w.l.o.g. that she starts at
$(0,0,\ldots,0)$ and then moves to $(0,1,\ldots,0)$. Breaker will not
allow her to visit any vertex whose first component is 1. When Walker
moves to $(0,x_2,x_3,\ldots,x_n)$, Breaker acquires the edge
$((0,x_2,x_3,\ldots,x_n),(1,x_2,x_3,\ldots,x_n))$. Breaker can acquire
the edge $((0,0,\ldots,0),(1,0,\ldots,0)$ on his last move, if not
before. It follows that at most $2^{n-1}$ vertices are visited by Walker.
This completes the proof of Theorem \ref{t.prwQ}.

\section{Proof of Theorem \ref{t.rwL}}
Here we will assume that Walker does a random walk on a graph $G$. When at a vertex $v$ she chooses a random neighbor $w$ for which the edge $(v,w)$ is not a Breaker edge.

Consider the first $t_0=4\a^{-1}n\log n$ moves. Let $G_t$ be the subgraph of $G$
induced by the edges not acquired by Breaker after $t$ moves. Let $L_t$ be the set of
vertices incident with more than $\a n/3$ Breaker edges after the
completion of $t$ moves by Breaker. Clearly $|L_t|\leq C_0\log n$,
where $C_0=24\a^{-2}$.

Let $v_t$ denote the current vertex being visited by Walker and let
$U_t$ denote the set of vertices that are not in $L_t$ and are
currently unvisited. Then
\begin{enumerate}[(a)]
\item If $v_t\notin L_t$ then the probability that Walker visits
 $U_t$ within two steps is at least $\b|U_t|/n$ where $\b=\a^3/36$,
 assuming that $|U_t|\geq \log n$. To see this let $Z=|N(v_{t+1})\cap
 U_t|$. Then $\E(Z)\geq \a|U_t|/2$. This is because  $v_t$ and any
 $w\in U_t$ have at least $\a n-2\a n/3=\a n/3$ common neighbors in $G_t$.  Thus, if $\bar{Z}=|U_t|-Z$ then $\E(\bar{Z})\leq (1-\a/3)|U_t|$. It follows from the Markov inequality that $\Pr(\bar{Z}\geq (1-\a^2/9)|U_t|)\leq \frac{1}{1+\a/3}$ and so
$\Pr(Z\geq \a^2|U_t|/9)\geq \frac{\a}{3+\a}$. Finally observe that $\Pr(v_{t+1}\in U_t\mid Z)\geq (Z-1)/n$, where we have subtracted 1 to account for Breaker's next move. 
\item We divide our moves up into periods $A_1,B_1,A_2,B_2,\ldots,$
  where $A_j$ is a sequence of moves taking place entirely outside
  $L_t$ and $B_j$ is a sequence of moves entirely within $L_t$. During
  a time period $A_j$, the probability this period ends is at most
  $\frac{24\a^{-2}\log n}{\a n/3}$. So the number of time periods is dominated by the binomial $Bin(4\a^{-1}n\log n,72\a^{-2}\log n/\a n)$ and so with probability $1-o(n^{-3})$ the number of periods is less than $300\a^{-4}\log^2n$.  
\item We argue next that 
\beq{long}
\text{with probability $(1-o(n^{-3})$ each $B_j$ takes up at most $O(\log^{6}n)$ moves.}
\eeq
Suppose that $B_j$ begins with a move from $v\notin L_t$ to $w\in
L_t$. Let $L^*=L_t\cup\set{v}$ and let $H^*$ denote the subgraph
induced by the edges contained in $L^*$ that have not been acquired by
Breaker. Walker's moves in period $B_j$ constitute a
random walk on (part) of the graph $H^*$. This is not quite a simple random random walk,
since $H^*$ changes due to the fact that Breaker can delete some the edges available to
Walker. Nevertheless, Walker will always be in a component of $H^*$
containing $v_t$. Now consider running this walk for $C_1\log^5n$ steps,
where $C_1$ is some sufficiently large constant. Observe that Breaker
can claim at most $C_0^2\log^2n$ edges inside this component of $H^*$. Hence there will be an interval of length
$C_2\log^3n,\,C_2=C_1/C_0^2$ where Breaker does not claim any edge
inside $H^*$. This means that in this interval we perform a simple
random walk on a connected graph with at most $(C_0+1)\log n$
vertices. If we start this interval at a certain vertex $x$, then we
are done if the random walk visits $v$. 
It follows from Brightwell and Winkler \cite{BW90} that the expected time
for the walk to visit $v$ can be bounded by $C_0^3\log^3n$. So, if
$C_2>2C_0^3$ then $v$ will be visited with probability at least 1/2. 

Suppose that time has increased from the time $t$ when $B_j$ began to $t'$ when $v$ is first re-visited. If $v\notin L_{t'}$ then $B_j$ is complete. If however $v\in L_{t'}$ then we know that $v$ is incident with at most $\a n/3+C_1\log^5n$ Breaker edges. So the probability that Walker leaves $L_{t'}$ in her next step is at least
\beq{bb1}
\frac{d_G(v)-(\a n/3+C_1\log^5n)}{d_G(v)}\geq \frac12.
\eeq 
So the probability that $B_j$ ends after $C_1\log^5n$ steps is at least 1/4. Suppose on the other hand that $B_j$ does not end and that we return to $v$ for $k$th time where $k\leq 20\log n$. The effect of this is to replace $C_1\log^5n$ in \ref{bb1} by $kC_1\log^5n$. This does not however affect the final inequality. So if $C_1$ is sufficiently large, the probability that $B_j$ does not end after $20C_1\log^6n$ steps is at most $(3/4)^{20\log n}=o(n^{-4})$.
Estimate \eqref{long}
follows immediately.
\item Combining the discussion in (b), (c) we see that w.h.p. 
$\card{\bigcup_jB_j}=O(\log^{8}n)$, which is negligible compared with $t_0$; i.e., Walker spends almost all of her time outside $L_{t_0}$. Let
$X_i,1\leq i\leq k=n-O(\log n)$ be the time needed to add the
$i$th vertex to the list of vertices visited by Walker. (Here we
exclude any time spent in $\bigcup_jB_j$). It follows from (a) that $X_i$
is dominated by a geometric random variable with probability of
success $\frac{(n-i)\b}{n}$. This is true regardless of
$X_1,X_2,\ldots,X_{i-1}$. So $\E(X_1+\cdots+X_k)\leq \frac{1}{\b}n\log
n$ and it is not difficult to show that $X_1+\cdots+X_k\leq
\frac{2}{\b}n\log n$ w.h.p. (We can use the Chebyshev inequality or
estimate the moment generating function of the sum and use Bernstein's method).
This completes the proof of Theorem \ref{t.rwL}.
\section{Proof of Theorem \ref{t.rwU}}
We assume that Walker chooses a vertex $a_0$ to start at and then Breaker chooses an edge to acquire.

Breaker's strategy will be to choose an arbitrary unvisited vertex
$v_1$ and protect it by always on his turn taking the edge $(v_1,w)$
where $w$ is the current vertex being visited by Walker, if
$(v_1,w)\in E(G)$. If Breaker has already acquired
$(v_1,w)$ or $(v_1,w)\notin E(G)$ then he will choose an unacquired edge incident with $v_1$. This continues until Breaker has acquired all of the edges incident with $v_1$. He then chooses $v_2$ and protects it. This continues until there are no unvisited vertices to protect.

After Breaker has protected $v_1,v_2,\ldots,v_{k-1}$ and while he is protecting $v_k$, Walker finds herself doing a random walk on a dense graph with $n-k$ vertices. Let the moves spent protecting $v_k$ be denoted by {\em round} $k$.

Fix $k=O(\log n)$ and let $\z_k$ be the number of unvisited, unprotected 
vertices when Breaker begins protecting $v_k$. It will take at most
$n-k-1$ more moves to protect $v_k$ and if $w$ is an unvisited, unprotected
vertex at the start of the round, then it remains unvisited with
probability at least $\brac{1-\frac{1}{\a n-k}}^{n-k-1}=e^{-1/\a}+O(1/(n-k))$.
It follows that $\E(\z_{k+1})\sim \z_k/e^{1/\a}$. Furthermore the random variable
$\z_{k+1}$ is tightly concentrated around its mean, if $\z_k\gg\log n$. We can use the Azuma-Hoeffding martingale tail inequality. (Changing one choice by Walker will change this random variable by at most one.)

It follows that w.h.p. $\z_k\sim ne^{-k/\a}$ for $k\leq (1-\e)\a\log
n$ where $0<\e<1$ is a positive constant. Thus Breaker will w.h.p. be
able to protect $(1-\e)\a\log n$ vertices and we can choose any
$c<\a$ in Theorem \ref{t.rwU}.
\end{enumerate}
This completes the proof of Theorem \ref{t.rwU}. 
\section{Further Questions}
Some natural questions spring to mind:
\begin{itemize}
\item How large a cycle can Walker make under the various conditions?
\item Suppose the goal is to visit as many edges as possible: what can be achieved under various game conditions?
\item Which subgraphs can Walker make? How large a clique can she
  make? Observe that PathWalker cannot even make a triangle.
\item What if we allow Walker to have $b$ moves to Breaker's one move?
\item What happens if Breaker is also a walker?
\end{itemize}
{\bf Acknowledgement:} We thank Dennis Clemens for pointing out an
error in an earlier version of Theorem \ref{t.rw11}. 

\end{document}